\documentclass[12pt]{article}
\usepackage{graphicx,amssymb}
\usepackage[bookmarks=true]{hyperref}
\usepackage[bookmarks=true]{hyperref}
\usepackage[bookmarks=true]{hyperref}
\newcommand{\bd}{\begin{displaymath}}
\newcommand{\be}{\begin{equation}}
\newcommand{\ba}{\begin{array}}
\newcommand{\ed}{\end{displaymath}}
\newcommand{\ee}{\end{equation}}
\newcommand{\ea}{\end{array}}
\newcommand{\espace}{\mbox{ }}

\newcommand{\Prob}{{\rm I\hspace{-0.8mm}P}}
\newcommand{\Exp}{{\rm I\hspace{-0.8mm}E}}
\newcommand{\indicator}[1]{{\mbox{\large\bf$1$}}_{#1}}

\def\N{\mathbb{N}}
\def\Z{\mathbb{Z}}
\def\R{\mathbb{R}}

\newcommand{\eqref}[1]{(\ref{#1})}

\newtheorem{theorem}{Theorem}[section]

\newtheorem{proposition}{Proposition}[section]

\newtheorem{lemma}{Lemma}[section]
\newtheorem{corollary}{Corollary}[section]

\newenvironment{proof}[2]{\espace\\{\em Proof of #1 \ref{#2}.}}{\hfill\mbox{$\square$}}
\begin{document}
\title{Supercriticality conditions for asymmetric zero-range process with sitewise disorder}
\author{C. Bahadoran$^{a}$, T. Mountford$^{b}$, K. Ravishankar$^{c}$, E. Saada$^{d}$}

\date{}
\maketitle
$$ \ba{l}
^a\,\mbox{\small Laboratoire de Math\'ematiques, Universit\'e Blaise Pascal, 63177 Aubi\`ere, France} \\
\quad \mbox{\small e-mail:
bahadora@math.univ-bpclermont.fr}\\
^b\, \mbox{\small Institut de Math\'ematiques, \'Ecole Polytechnique F\'ed\'erale, 
 
Lausanne, Switzerland
} \\
\quad \mbox{\small e-mail:
thomas.mountford@epfl.ch}\\
^c\, \mbox{\small Dep. of Mathematics, SUNY,
College at New Paltz, NY, 12561, USA} \\
\quad \mbox{\small e-mail:
 ravishak@newpaltz.edu}\\
^d\, \mbox{\small CNRS, UMR 8145, MAP5,
Universit\'e Paris Descartes,

Sorbonne Paris Cit\'e, France}\\
\quad \mbox{\small e-mail:
Ellen.Saada@mi.parisdescartes.fr}\\ \\
\ea
$$
\begin{abstract}
We discuss necessary and sufficient conditions for the convergence of disordered asymmetric 
zero-range process to the critical invariant measures.
\end{abstract}
\section{Introduction}
\label{sec_intro}
We first discuss the aims of this paper and our results without formal statements or definitions, 
which will be given in the following section.
We study inhomogenous zero range processes (to be defined in the next section) and the question 
of convergence to upper invariant measures.
Zero range processes  (see e.g. \cite{and})  are conservative particle systems, a class containing 
the much studied exclusion processes (\cite{lig}).  In this case, to our best knowledge necessary and sufficient
conditions are hard to find.  In the setting of homogenous processes with product measures for 
initial conditions one can profit from hydrodynamic limits and argue convergence as in \cite{L}.
This elegant approach is not relevant here as our systems are not translation invariant and we are 
faced with deterministic initial conditions.  The lack of translation invariance also hampers the 
approach of \cite{EG}.  The exclusion process starting from fixed initial conditions but having an 
asymptotic density is treated in \cite{BM} and gives a robust criterion for weak convergence but a 
necessary and sufficient condition for weak convergence to a given equilibrium seems difficult.  
Thus it is of interest to be able to, in some reasonable circumstances, give a necessary and 
sufficient condition for weak convergence to a particular equilibrium.   Other works which 
investigate convergence to equilibrium of zero range processes, albeit from a different standpoint, 
include \cite{GG} and \cite {JLQY}.
This article considers one of the principal results of \cite{afgl}
concerning the totally asymmetric nearest neighbour zero range process
having inhomogeneous ``service rates" $\alpha(x),  x\in \Z$ so that at rate 
$\alpha(x)$  one of the particles currently at site $x $ (if any) will move to site $x+1$.
It was postulated that \\ \\
\textit{(i)} there existed $ 0 < c < 1$ so that for every $x, \ \alpha(x) \in (c,1]$, \\
\noindent
\textit{(ii)} for every flux value $\lambda \in [0, c)$, there existed a particle density, 
$\overline{R}( \lambda ) $, \\
\noindent
\textit{(iii)} the liminf of $\alpha (x) $ as $x$ tended to - infinity was equal to $c$, \\
\noindent
\textit{(iv)} the (increasing) limit of  $\overline{R}( \lambda )$ as $\lambda$ increased to $c$, 
denoted $\overline{R}(c)$, was finite.\\ \\
Under condition \textit{(i)}, for each $0 \leq \lambda \leq c$, the measure $\nu_\lambda $, 
under which $\eta(x) $ were independently distributed as Geometric random variables with
parameter ${\lambda }/\alpha(x)$, was an equilibrium for the zero range process.  
Property \textit{(ii)} ensured that for each  
$\lambda\in[0,c)$, $\nu_\lambda$-a.s.,
$$
\lim_{n \rightarrow \infty} \frac{1}{n} \sum_{x=-n}^ {-1} \eta(x) \ = \
\lim_{n \rightarrow \infty} \frac{1}{n} \sum_{x=-n}^ {-1} \frac{\lambda}{\alpha(x) - \lambda }   
=: \overline{R}(\lambda )
$$
A motivating special case was where the values $\alpha (x)$ were i.i.d. (or ergodic) 
random variables with marginal law $Q$.  In this case the assumptions amount to law $Q$
being supported by $(c,1]$ and satisfying 
$$
\rho_c:=\int_c^ 1 \frac{c}{\alpha - c } Q(d \alpha ) \ < \ \infty 
$$
Given an equilibrium $\nu_\lambda$, it is natural to ask for which initial configurations, 
$\eta_0$, is it the case that for the zero range process $(\eta_t)_{t \geq 0}$
beginning at $\eta_0$, $\eta_t $ converges in distribution to $\nu_\lambda $ as $t\to +\infty$.  
The paper \cite{afgl} concerns the totally asymmetric constant rate zero-range process, defined by
\be\label{special_zrp}
g(n)=\indicator{\{n\geq 1\}},\quad p(1)=1
\ee
The result of \cite{afgl} is that, if
$$
 \liminf _{n \rightarrow \infty}  \frac{1}{n} \ \sum_{x=-n} ^{-1} \eta_0(x) \ > \rho_c,
$$
then $\eta_t $ converges in distribution to $\nu_c$.\\ \\
Our objective is to complete this statement in the following ways.
First, we consider rate functions $g(.)$ that increase to a finite limiting value 
and satisfying a weak concavity condition (H) (to be stated in the next section).
Next, we allow non totally asymmetric nearest neighbour random walk kernels.
We address the following natural questions.
On the one hand, we prove that the condition 
\be\label{cond_init}
\liminf _{n \rightarrow \infty}  \frac{1}{n} \ \sum_{x=-n} ^{-1} \eta_0(x)  \geq \rho_c
\ee
is necessary for weak convergence.
On the other hand, we prove that condition \eqref{cond_init} does not imply 
convergence for jump kernels that are not nearest-neighbour. Note that the latter 
result is in sharp contrast with known result for the homogeneous simple exclusion process.
The previous results indicate that it is reasonable to seek a generalization of the result 
of \cite{afgl} to nearest-neighbour jump kernels and more general rate functions $g(.)$. 
In this paper, we will establish a 
 general upper bound, and provide some ideas of how to prove a lower bound, using new hydrodynamic 
 limit results established in \cite{bmrs2}. These new ideas show that convergence is actually implied 
 by the weaker supercriticality condition \eqref{cond_init},
which is another improvement of the result of \cite{afgl}. This leads to the conclusion that 
\eqref{cond_init} is a necessary and sufficient condition.
To our knowledge this is the first such condition given for conservative systems where the kernel 
governing particle motion has nonzero mean.  The article \cite{BM} gives a robust condition for 
convergence to a translation invariant extremal equilibrium
for exclusion processes but, seemingly, finding a necessary  and sufficient condition is more subtle.
\section{Notation and results}
\label{sec_results}
In the sequel,  $\R$ denotes the set of real numbers, $\Z$  the set of signed integers 
and $\N=\{0,1,\ldots\}$ the set of nonnegative integers. For $x\in\R$,  $\lfloor x\rfloor$  
denotes the integer part of $x$, that is largest integer $n\in\Z$ such that $n\leq x$. 
The notation $X\sim\mu$ means that a random variable $X$ has probability distribution $\mu$.\\ \\
 Fix some $c\in(0,1)$. An environment (or disorder) is a $(c,1]$-valued sequence 
 $\alpha=(\alpha(x),\,x\in\Z)$. The set of environments
is denoted by ${\bf A}:=(c,1]^\Z$. From now on, we assume that 
\be
\label{not_too_sparse}
\liminf_{x\to-\infty}\alpha(x)=c
\ee
Let $g:\N\to[0,+\infty)$ be a nondecreasing function such that
\be\label{properties_g}
g(0)=0<g(1)\leq \lim_{n\to+\infty}g(n)=:g_\infty<+\infty
\ee
We extend $g$ to  $\overline{\N}:=\N\cup\{+\infty\}$  by setting $g(+\infty)=g_\infty$.
Without loss of generality, we henceforth assume $g_\infty=1$.
Let ${\bf X}:=\overline{\N}^\Z$ denote the set of particle configurations.
  A configuration is of the form $\eta=(\eta(x):\,x\in\Z)$ where $\eta(x)\in\overline{\N}$ for each $x\in\Z$.
Let $p(.)$ be a probability measure on $\Z$.
We consider the Markov process  $(\eta_t^\alpha)_{t\geq 0}$
on ${\bf X}$ with generator given for any cylinder function $f:{\bf X}\to\R$ by
\be\label{generator}
L^\alpha f(\eta)  =  \sum_{x,y\in\Z}\alpha(x)
p(y-x)g(\eta(x))\left[
f\left(\eta^{x,y}\right)-f(\eta)
\right]\ee
where, if $\eta(x)>0$, $\eta^{x,y}:=\eta-\delta_x+\delta_y$ denotes the new 
configuration obtained from $\eta$ after a particle has jumped from $x$ to $y$
(configuration $\delta_x$ has one particle at $x$ and no particle elsewhere; 
addition of configurations is meant coordinatewise).  In cases of infinite particle number,  
the following interpretations hold: $\eta^{x,y}=\eta-\delta_x$ if 
$\eta(x)<\eta(y)=+\infty$, $\eta^{x,y}=\eta+\delta_y$ if $\eta(x)=+\infty>\eta(y)$,
$\eta^{x,y}=\eta$ if $\eta(x)=\eta(y)=+\infty$. \\ \\
 Because $g$ is nondecreasing, $(\eta_t^\alpha)_{t\geq 0}$ is an attractive process (\cite{and}). \\ \\
For $\lambda<1$, we define the probability measure $\theta_\lambda$ on $\N$
 by \label{properties_a}
\begin{equation}\label{eq:theta-lambda}
\theta_\lambda(n):=Z(\lambda)^{-1}\frac{\lambda^n}{g(n)!},\quad n\in\N
\end{equation}
where $g(n)!=\prod_{k=1}^n g(k)$ for $n\geq 1$, $g(0)!=1$, and $Z(\lambda)$ is the normalizing factor:
$$
Z(\lambda):=\sum_{n=0}^{+\infty}\frac{\lambda^n}{g(n)!}
$$
We extend $\theta_\lambda$ into a probability measure on $\overline{\N}$ by setting $\theta_\lambda(\{+\infty\})=0$.
For $\lambda\leq c$, we denote by $\mu_\lambda^\alpha$ the invariant measure of 
$L^\alpha$ defined (see e.g. \cite{bfl}) as the product measure on $\bf X$ with one-site marginal
$\theta_{\lambda/\alpha(x)}$. 
Since $(\theta_\lambda)_{\lambda\in [0,1)}$ is an exponential family, 
the mean value of $\theta_\lambda$, given by 
$$
R(\lambda):=\sum_{n=0}^{+\infty}n\theta_\lambda(n)
$$
is a $C^\infty$ increasing function from $[0,1)$ to $[0,+\infty)$.
The quenched mean particle density at $x$ under $\mu_\lambda^\alpha$ is defined by
$$
R^\alpha(x,\lambda):=\Exp_{\mu_\lambda^\alpha}[\eta(x)]=R\left(\frac{\lambda}{\alpha(x)}\right)
$$
In order to  define a notion of critical density, we assume existence of an annealed 
mean density to the left of the origin:
\be\label{average_afgl}
\overline{R}(\lambda):=\lim_{n\to+\infty}n^{-1}\sum_{x=-n}^0 R\left(\frac{\lambda}{\alpha(x)}\right)\quad\mbox{exists for every }\lambda\in[0,c)
\ee
The function $\overline{R}$ is increasing and  $C^\infty$ function on $[0,c)$ (see Lemma \ref{lemma_properties_flux} below).
We define the critical density by 
\be\label{other_def_critical}
\rho_c:=\overline{R}(c-):=\lim_{\lambda\uparrow c}\overline{R}(\lambda)\in[0,+\infty]
\ee
Note that formally, one is tempted to define $\overline{R}(c)$ by plugging 
$\lambda=c$ into \eqref{average_afgl}.
However the corresponding limit may have a different value than $\rho_c$, or even not exist.
In fact, it can be made non-existent or given any value in $[\rho_c,+\infty)$ by modifying 
the $\alpha(x)$ for $x$ 
in a zero-density subset of $\N$, an operation which does not modify the value 
\eqref{average_afgl} nor the value of $\rho_c$. 
The only stable property with respect to such change is that
$$
\liminf_{n\to+\infty}\frac{1}{n+1}\sum_{x=-n}^0 R\left(\frac{c}{\alpha(x)}\right)\geq\rho_c
$$
It is thus natural, and it will be our choice in the sequel, to extend $\overline{R}$ 
by continuity to $[0,c]$ by defining 
\be
\label{extend_R}
\overline{R}(c):=\overline{R}(c-)=\rho_c
\ee
One additional assumption for Theorem \ref{convergence_maximal} and Proposition 
\ref{prop_main_work} below will be
finiteness of the critical density:
\be\label{assumption_pt}\rho_c<+\infty\ee
We can now state our results. With the exception of Theorem \ref{prop_counter}, we will consider   
a nearest-neighbour jump kernel with non-zero drift, that is 
\be\label{nn_hyp}
p(1)=p\in(1/2,1],\quad p(-1)=q:=1-p
\ee
This is not a technical artefact, see   
Theorem \ref{prop_counter} below. 
In the forthcoming statements, $\eta_0\in\N^\Z$ denotes the initial particle configuration,
and $(\eta_t^\alpha)_{t\geq 0}$ the evolved quenched process with generator \eqref{generator} starting from $\eta_0$ in the environment $\alpha\in{\bf A}$. 
Our general problem is to determine whether, and for what kernels $p(.)$, rate functions 
$g(.)$ and environments $\alpha(.)$, the supercriticality condition \eqref{cond_init}, 
with $\rho_c$ defined by \eqref{other_def_critical}, is necessary and sufficient for the convergence
\be\label{eq:convergence_critical} 
\eta_t\longrightarrow\mu^\alpha_c\mbox{ in distribution as }t\to+\infty
\ee
The study of \eqref{eq:convergence_critical}  can be decomposed into an upper bound 
and a lower bound. For the former, we prove in Section \ref{sec_afgl} the following result.
\begin{theorem}\label{proposition_afgl}
Assume \eqref{not_too_sparse} and \eqref{nn_hyp}.
 every $\eta_0\in\N^\Z$ and every bounded local nondecreasing function $h:{\bf X}\to\R$,  
\be\label{upper_bound}
\limsup_{t\to\infty}\Exp h(\eta_t^\alpha)\leq\int_{\bf X} h(\eta)d\mu_c^\alpha(\eta)
\ee
\end{theorem}
The upper bound \eqref{upper_bound} was established in \cite{fs} for i.i.d. environments, 
in any space dimension and for any (not necessarily nearest-neighbour)
jump kernel $p(.)$ with nonzero drift, under an additional assumption on the initial configuration:
\be
\label{cond_fs}
\sum_{n\in\N}e^{-\beta n}\sum_{x:\,|x|=n}\eta_0(x)<+\infty,\quad\forall\beta>0
\ee
In one dimension with nearest-neighbour jumps, we  
prove Theorem \ref{proposition_afgl} without  assumption
 \eqref{cond_fs}, and under the  weak  
 assumption \eqref{assumption_afgl}
 on the environment. This is done by extending an argument used in
\cite{afgl} in the special case $p=1$ and $g(n)=\indicator{\{n\geq 1\}}$.\\ \\
Our next result shows that the supercriticality condition \eqref{cond_init} is in fact 
necessary if we slightly strengthen \eqref{not_too_sparse} by assuming that
the slow sites (i.e. with rates close to $c$) are not too sparse in the following sense: 
there exists a sequence $(x_n)_{n\in\N}$ of sites such that
\be\label{assumption_afgl}
x_{n+1}<x_n<0,\quad \lim_{n\to +\infty}\frac{x_{n+1}}{x_n}=1,\quad \lim_{n\to+\infty}\alpha(x_n)=c
\ee
\begin{theorem}\label{prop_nec}
Assume  \eqref{average_afgl}, 
\eqref{nn_hyp} and \eqref{assumption_afgl}. Assume further that $\eta_0$ satisfies
\be\label{eq_non16}
\rho=\liminf_{n\to\infty}n^{-1}\sum_{x=-n}^0\eta_0(x)< \rho_c
\ee
 Then 
$\eta_t^\alpha$
does not converge in distribution to $\mu_c^\alpha$ as $t\to+\infty$.
\end{theorem}
We shall prove Theorem \ref{prop_nec}
in Section \ref{sec_nec}. \\ \\
Before turning eventually  to the question of convergence \eqref{eq:convergence_critical} 
under supercriticality condition \eqref{cond_init}, we show that this property cannot hold 
in general beyond the nearest-neighbour case \eqref{nn_hyp}. Indeed, the following theorem, 
proved in Section \ref{sec:proof_counter}, provides a family of counterexamples for jump 
kernels $p(.)$ not satisfying the nearest-neighbour assumption.
\begin{theorem}\label{prop_counter}
Assume 
\eqref{not_too_sparse} and  \eqref{average_afgl}.
Assume further that the jump kernel $p(.)$ is totally asymmetric and $p(1)<1$. 
Then there exists $\eta_0\in\N^\Z$ satisfying \eqref{cond_init}, such that 
$\eta^\alpha_t$ does not converge in distribution to $\mu_c^\alpha$ as $t\to +\infty$.
\end{theorem}
For the purpose of convergence, we need to introduce the following weak convexity assumption:\\ \\ 
(H) For every $\lambda\in[0,c)$, $\overline{R}(\lambda)-\overline{R}(c)-(\lambda-c)\overline{R}^{'+}(c)>0$\\ \\
where $\overline{R}(c)$ is defined by \eqref{extend_R}, and
\be\label{def_upperdiff}
\overline{R}^{'+}(c):=\limsup_{\lambda\to c}\frac{\overline{R}(c)-\overline{R}(\lambda)}{c-\lambda}
\ee
is the left-hand derivative at $c$ of the convex envelope of $\overline{R}$
(notice that our assumptions do not imply existence of the derivative $\overline{R}'(c)$).\\ \\
For instance, if $R$ is strictly convex, then for any environment 
satisfying \eqref{average_afgl}--\eqref{assumption_afgl}, 
$\overline{R}$ is strictly convex (see Lemma \ref{lemma_properties_flux} below), 
and thus (H) satisfied. A sufficient condition for $R$ to be strictly convex 
(see \cite[Proposition 3.1]{bs}) is that 
$$ 
n\mapsto g(n+1)-g(n)\mbox{ is a nonincreasing function}
$$
The following result is established in \cite{bmrs2}.
\begin{theorem}
\label{convergence_maximal}
Assume  \eqref{average_afgl}, \eqref{assumption_pt}, \eqref{nn_hyp}, 
\eqref{assumption_afgl} and (H). Then \eqref{cond_init} implies \eqref{eq:convergence_critical}.
\end{theorem}
Given Theorem \ref{proposition_afgl}, Theorem \ref{convergence_maximal} 
requires the proof of the following lower bound:
\begin{proposition}\label{prop_main_work}
Assume  \eqref{average_afgl}, \eqref{assumption_pt}, \eqref{nn_hyp}, \eqref{assumption_afgl} and (H). 
Then the following holds: 
for any $\eta_0\in{\N^\Z}$  satisfying \eqref{cond_init}, and every bounded local nondecreasing 
function $h:{\bf X}\to\R$,
$$
\liminf_{t\to\infty}\Exp h(\eta_t^\alpha)\geq\int_{\bf X} h(\eta)d\mu_c^\alpha(\eta)
$$
\end{proposition}
The ideas of the proof of Proposition \ref{prop_main_work}, and in particular 
the role of assumption (H), are explained in Section \ref{sec_ideas}.
\section{Preliminaries}\label{sec_prel}
In this section, we introduce some material that will be used in the sequel.
\subsection{Harris construction}
We define the graphical construction of the quenched process on a probability space $(\Omega,\mathcal F,\Prob)$.
A generic element $\omega$ - called a Harris system (\cite{har}) - of $\Omega$  is a  point measure of the form
\be\label{def_omega}
\omega(dt,dx,du,dz)=\sum_{n\in\N}\delta_{(T_n,X_n,U_n,Z_n)}
\ee
 on $(0,+\infty)\times\Z\times(0,1)\times\Z$, where $\delta$ denotes Dirac measure.
Under the probability measure $\Prob$, $\omega$ is a Poisson measure with intensity 
$dtdx\indicator{[0,1]}(u)du\, p(z)dz$. 
 An alternative interpretation of this random point measure is that we have three mutually 
 independent families of independent 
random variables $(D^x_k)_{x\in\Z,\,k\in\N}$, $(U^x_k)_{x\in\Z,\,k\in\N}$ and
$(Z^x_k)_{x\in\Z,\,k\in\N}$,  such that $D^x_k$ has exponential distribution with 
parameter $1$, $U_k^x$ has uniform distribution on $(0,1)$,
$Z_x^k$ has distribution  $p(.)$, and that if we set
\be\label{poisson_one}
T^x_k:=\sum_{j=0}^k D^x_j,
\ee
then, $\Prob$-a.s.,
\be\label{notation_harris}
\omega(dt,dx,du,dz)=\sum_{x\in\Z}\sum_{k\in\N}\delta_{(T_k^x,x,U_k^x,Z^x_k)}
\ee
On  $(\Omega,\mathcal F,\Prob)$,  a c\`adl\`ag process $(\eta_t^\alpha)_{t\geq 0}$ 
with generator \eqref{generator} and initial configuration $\eta_0$ can be constructed
in a unique way so that at each time $t=T^x_k$, if 
$U^x_k\leq \alpha(x)g\left[\eta_{t-}(x)\right]$ (which implies $\eta_{t-}(x)>0$, cf. \eqref{properties_g}), 
then one of the particles at $x$ jumps to $x+Z^x_k$, whereas nothing occurs outside times $T^x_k$.\par
For details on this graphical construction, we refer to \cite{bgrs4}. 
When  necessary, random initial conditions are constructed on an auxiliary probability space $\Omega_0$ 
equipped with a probability measure $\Prob_0$. \\ \\
 Expectation with respect to $\Prob$ (resp. $\Prob_0$) is denoted by $\Exp$ (resp. $\Exp_0$). 
The product space $\Omega_0\times\Omega$ is equipped with the product measure and 
$\sigma$-fields (thus  environment, initial particle configuration and Harris system are mutually  independent). 
Joint expectation with respect to the product 
measure  is denoted by $\Exp_0\Exp$.\\ \\
In the sequel, we shall have to couple different processes with different 
(possibly random) initial configurations, and possibly different environments.
Such couplings will be realized on $\Omega_0\times\Omega$ by using the same Poisson clocks for all processes.
\subsection{Finite propagation}
The following version of {\em finite propagation property} will be used repeatedly in the sequel. 
\begin{lemma}\label{lemma_finite_prop}
For each $W > 1$, there exists $b =  b(W) >  0$ such that for large enough $t$, if 
$ \zeta_0 $ and $\zeta' _ 0 $ are any two particle configurations that agree on an 
interval $(x,y)$, then, outside probability $e^{-bt}$,
$$
\zeta_s(u) = \zeta'_s(u) \quad\mbox{for all }0 \leq s \leq t \mbox{ and } u\in (x+Wt,y-Wt)
$$ 
where $(\zeta_t)_{t\geq 0}$ and $(\zeta'_t)_{t\geq 0}$ denote processes with generator
 \eqref{generator} coupled through the Harris construction,
with respective initial configurations $\zeta_0$ and $\zeta'_0$.
\end{lemma}
\mbox{}
\begin{proof}{Lemma}{lemma_finite_prop}
Given the Harris system and a point $x \in \Z$, we can define the process $(B^{x,+}_t )_{t \geq 0 }$ 
to be the (c\`adl\`ag) Poisson process 
beginning at value $x$ at time $0$ such that $B^{x,+}_.$ jumps to the right from $y $ to $y+1$ 
at time $s$  if and only if $B^{x,+}_{s-}=y$, and $ s = T^y_k$ for some positive integer $k$, 
with $(T^y_k)_{k\in\N}$ introduced 
in \eqref{poisson_one}.
We similarly define the decreasing Poisson process $\left(B^{x,-}_t\right)_{t\geq 0}$. 
It follows from the graphical construction that,  if we have two configurations 
$\eta_0 $ and $ \xi_0 $ which agree on interval $(x,y)$, then $\zeta_s $ and $\zeta'_s$ will agree 
on spatial interval   $(B^{x,-}_s, B^{y,+}_s)$ for all $0 \leq s \leq t$.  
 Given that the defined processes (or their negatives) are rate one 
Poisson processes we deduce the result.
\end{proof}
\subsection{Currents}
Let $x_.=(x_s)_{s\geq 0}$ denote a $\Z$-valued piecewise constant c\`adl\`ag   path such 
that $\vert x_s-x_{s-}\vert\leq 1$ for all $s\geq 0$. In the sequel we will only use 
deterministic paths $(x_.)$, 
hence we may assume that $x_.$ has no jump time in common with the Harris system used for the particle dynamics.
We denote by $\Gamma_{x_.}(t,\eta)$ the rightward current across the path $x_.$ 
up to time $t$ in the quenched process  $(\eta_s^\alpha)_{s\geq 0}$  starting from $\eta$ 
in environment $\alpha$, that is, 
 the number of times a particle jumps from $x_{s-}$ to $x_{s-}+1$ (for $s\le t$), 
 minus the number of times a particle jumps from $x_{s-}+1$ to $x_{s-}$, 
 minus  or plus (according to whether the jump is to the right or left)  
  the number of particles at $x_{s-}$ if $s$ is a jump time of $x_.$.
 If  $\sum_{x>x_0}\eta(x)<+\infty$,  we also have
 \be\label{current}
 \Gamma_{x_.}(t,\eta)=\sum_{x>x_t}\eta_t^\alpha(x)-\sum_{x>x_0}\eta(x)
 \ee
For $x_0\in\Z$, we will write $\Gamma_{x_0}$ to denote the current across the 
fixed site $x_0$; that is, $\Gamma_{x_0}(t,\eta):=\Gamma_{x_.}(t,\eta)$,
where $x_.$ is the constant path defined by $x_t=x_0$ for all $t\geq 0$.
\\ \\
The following results will be important tools to compare currents.
For a particle configuration $\zeta\in{\bf X}$ and a site $x_0\in \Z$, we define
\be\label{def_cfd}
F_{x_0}(x,\zeta):=\left\{
\ba{lll}
\sum_{y=1+x_0}^{x}\zeta(y) & \mbox{if} & x>x_0\\ \\
-\sum_{y=x}^{x_0}\zeta(y) & \mbox{if} & x\leq x_0
\ea
\right.
\ee
Let us couple two processes $(\zeta_t)_{t\geq 0}$ and $(\zeta'_t)_{t\geq 0}$ in the usual 
way through the Harris construction. 
\begin{lemma}
\label{lemma_current}
\be\label{current_comparison}\Gamma_{x_.}(t,\zeta_0)-\Gamma_{x_.}(t,\zeta'_0)\geq -\left(0\vee\sup_{y\in\Z}\left[F_{x_0}(y,\zeta_0)-F_{x_0}(y,\zeta'_0)\right]\right)\ee
\end{lemma}
\begin{proof}{Lemma}{lemma_current} 
Without loss of generality, we assume $x_0=0$. We shall simply write $F$ for $F_{0}$ and $\Gamma$ for $\Gamma_{x_.}$.\\ \\
We label particles of each system increasingly from left to right.
 We denote by $\sigma_i(t)\in\Z\cup\{+\infty,-\infty\}$, resp. $\sigma'_i(t)\in\Z\cup\{+\infty,-\infty\}$, 
the position at time $t$ of the $\zeta$-particle with label $i\in\Z$.
 resp. $\zeta'$-particle with label $i$. This labeling is unique if we impose the following conditions: 
\textit{(a)} $\sigma_i(0)\leq\sigma_{i+1}(0)$, resp. $\sigma'_i(0)\leq\sigma'_{i+1}(0)$, for every $i\in\Z$,  
\textit{(b)} $\sigma_{-1}(0)\leq x_0=0<\sigma_0(0)$ and $\sigma'_{-1}(0)\leq x_0=0<\sigma'_0(0)$, 
\textit{(c)}
 if $n_+:=\sum_{y>0}\zeta_0(y)<+\infty$, then $\sigma_i(0)=+\infty$ for $i\geq n_+$, 
 (d) if $n_-=\sum_{y\leq 0}\zeta(y)<+\infty$, then $\sigma_i(0)=-\infty$ for $i<-n_-$.
 In the sequel, for notational simplicity, we simply write $\sigma_i$ and $\sigma'_i$ 
 instead of $\sigma_i(0)$ and $\sigma'_i(0)$.\\ \\
 The motion of labels is deduced from the initial labeling and Harris construction as follows. 
 Whenever a particle in one of the processes jumps to the right (resp. left),
 the highest (resp. lowest) label occupying this site is moved to the right (resp. left), 
 so that at its new location it becomes the particle with lowest (resp. highest) label.
 Let $k=0\vee\sup_{y\in\Z}\left[F(y,\zeta)-F(y,\zeta')\right]$. \\ \\
 First, we show that the definition of  $F$ implies $\sigma_n\geq\sigma'_{n-k}$ for all $n\in\Z$.
 Indeed,
 assume first $n\geq 0$, and let $y=\sigma_n>0$. By definition of $F$, $F(y,\zeta)-1$ is 
 the highest label of a $\zeta$-particle at $y$,
 thus $F(y,\zeta)\geq n+1$. By definition of $k$, $F(y,\zeta')-1\geq F(y,\zeta)-k-1\geq n+1-k-1=n-k$, 
 and this is the highest label of a $\zeta'$-particle at $y$. Hence $\sigma'_{n-k}\leq y=\sigma_{n}$.
 Assume now $n<0$, and  let $y=\sigma'_n\leq 0$. By definition of $F$, $F(y,\zeta')$ is 
 the lowest label of a $\zeta'$-particle at $y$,
 thus $F(y,\zeta')\leq n$. By definition of $k$, $F(y,\zeta)\leq F(y,\zeta)+k\leq n+k$, 
 and this is the lowest label of a $\zeta'$-particle at $y$. Hence $\sigma_{n+k}\geq y=\sigma'_{n}$.\\ \\
 Next, we show that $\sigma_n\geq\sigma'_{n-k}$ for all $n\in\Z$ implies
 \be\label{later_order}\sigma_n(t)\geq\sigma'_{n-k}(t),\quad\forall n\in\Z\ee
 Setting $\tilde{\sigma'}_n=\sigma'_{n-k}$, we define another increasing labeling of $\zeta'$, 
 and the definition of label's motion implies that $\tilde{\sigma'}_n(t)=\sigma'_{n-k}(t)$ 
 for all $t\geq 0$. Therefore we have to show that $\sigma_n\ge\tilde{\sigma'}_n$ for all  $n\in\Z$ implies
 $\sigma_n(t)\ge\tilde{\sigma'}_n(t)$ for all  $n\in\Z$.
 It is enough to show that this ordering is preserved by the Harris construction each time a jump occurs.
 Assume that at time $t$ a potential jump to the right occurs in the Harris construction. 
 We have to verify the following for any $n\in\Z$:\\ \\
\textit{(1) If a particle at $y\in\Z$ in $\zeta'$ jumps right, and the particle in $\zeta$ 
with the same label is also at $y$, then  the latter particle also jumps right.}\\ \\
 Indeed, since we are assuming $\sigma_n \geq \sigma'_n$ for all $n \in \mathbb{Z}$ 
 we have $F(z,\zeta) \leq F(z,\zeta')$ for all $z \in \mathbb{Z}$.
 Suppose $\sigma_n = \sigma'_n=y$ for some $n\in\Z$, and that $n$ is highest label of $\zeta'$ particles at $y$.
 That is $n = F(y,\zeta') \geq F(y,\zeta)$. If $F(y,\zeta') > F(y,\zeta)$ then $F(y,\zeta) \leq n-1$.
 This implies that $\sigma_n \geq y+1$. Since we assumed $\sigma_n =y$ we conclude that 
 $F(y,\zeta') =F(y,\zeta)=n$.
 Therefore $n$ is the highest label of $\zeta$ particles at $y$.
 Also, $F(y,\zeta) \leq F(y,\zeta')$ for all $y\in \mathbb{Z}$ and  $F(y,\zeta') =F(y,\zeta)$ 
 implies $\zeta(y) \geq \zeta'(y)$.
 Since $g$ is increasing, if a jump occurs from $\zeta'(y)$, then a jump from $\zeta(y)$ must occur.
 \\ \\
\textit{(2) If a particle at $y\in\Z$ in $\zeta$ jumps left, and the particle in $\zeta'$ 
with the same label is also at $y$, then  the latter particle also jumps left.}\\ \\
 Indeed, suppose  $n$ is the smallest label of $\zeta$ particles at $y$.
 Then $n = F(y-1,\zeta)+1$ or $F(y-1,\zeta) = n-1$. If $F(y-1,\zeta') > F(y-1,\zeta)= n-1$ then $\sigma'_n <y$.
 Since we assumed $\sigma'_n= y$, it follows that $F(y-1,\zeta') = F(y-1,\zeta) = n-1$.
 Again, since $\sigma'_n =y$, we have that $n$ is the smallest label of $\zeta'$ particles at $y$.
 Again from the fact that $F(y-1,\zeta') = F(y-1,\zeta)$ and $F(z,\zeta) \leq F(z,\zeta')$ 
 for all $z \in \mathbb{Z}$, it follows that $\zeta(y) \leq \zeta'(y)$.
 Therefore, since $g$ is increasing, we have that if a jump to the left from $y$ occurs in $\zeta$, 
 then it occurs also for $\zeta'$.\\ \\
 To conclude the proof, we use the following definition of current in terms of labels: 
 \begin{eqnarray*}
 \label{current_labels_1}
 \Gamma(t,\zeta) & = & -\inf\{m\in\Z\cap(-\infty,0):\,\sigma_m(t)>x_t\}\mbox{ if }\Gamma(t,\zeta)>0\\
 \label{current_labels_2}
 \Gamma(t,\zeta) & = & -\sup\{n\in\Z\cap[0,+\infty):\,\sigma_n(t)\leq x_t\}\mbox{ if }\Gamma(t,\zeta)\leq 0
 \end{eqnarray*}
 Since $\Gamma(t,\zeta)\geq\Gamma(t,\zeta')-k$ holds if $\Gamma(t,\zeta)\geq 0$ and 
 $\Gamma(t,\zeta')\leq 0$, it is enough to consider the following cases.
 Assume first that $m:=-\Gamma(t,\zeta')<0$, hence $\sigma'_{m-1}(t)\leq x_t<\sigma'_m(t)$. 
 By \eqref{later_order}, $\sigma_{m+k}(t)\geq \sigma'_m(t)>x_t$. Thus 
 $\Gamma(t,\zeta)\geq-(m+k)=\Gamma(t,\zeta')-k$.
 Next, assume that $n:=-\Gamma(t,\zeta)>0$, hence $\sigma_{n-1}(t)\leq x_t<\sigma_n(t)$.
 By \eqref{later_order}, $\sigma'_{n-1-k}(t)\leq \sigma_{n-1}(t)\leq x_t$. 
 Thus $\Gamma(t,\zeta')\leq-n+k=\Gamma(t,\zeta)+k$.
\end{proof}
\mbox{}\\ \\
\begin{corollary}\label{corollary_consequence}
For $y\in\Z$, define the configuration 
\be\label{conf-max}
\eta^{*,y}:=(+\infty)\indicator{(-\infty,y]\cap\Z}
\ee
Then, for every $z\in\Z$ such that $y\leq z$ and every $\zeta\in{\bf X}$,
\be\label{eq_corollary_consequence}
\Gamma_z(t,\zeta)\leq\Gamma_z(t,\eta^{*,y})+\indicator{\{y<z\}}\sum_{x=y+1}^z\zeta(x)
\ee
\end{corollary}
\begin{proof}{Corollary}{corollary_consequence}
Note that, for every $\zeta'\in{\bf X}$  and $u\in\Z$, we have
\be\label{ext_conf}
F_{y}(u,\eta^{*,y})\leq F_{y}(u,\zeta')-\indicator{\{y<u\}}\zeta'(u)
\ee 
and 
\be\label{43bis}
F_y(u,\zeta')=F_z(u,\zeta')+\indicator{\{y<z\}}\sum_{x=y+1}^z \zeta'(x)+\indicator{\{y<u<z\}}\zeta'(u)
\ee 
It follows that
\begin{eqnarray}
F_z(u,\eta^{*,y}) & = & F_y(u,\eta^{*,y})\nonumber\\
& \leq & F_y(u,\zeta)-\indicator{\{y<u\}}\zeta(u)\nonumber\\
& = & F_z(u,\zeta)+\indicator{\{y<z\}}\sum_{x=y+1}^z\zeta(x)+\indicator{\{y<u<z\}}\zeta(u)
 -\indicator{\{y<u\}}\zeta(u)\nonumber\\
& \leq & F_z(u,\zeta)+\indicator{\{y<z\}}\sum_{x=y+1}^z\zeta(x)\label{43plus43bis}
\end{eqnarray}
where we used \eqref{ext_conf} on the second line with $\zeta'=\zeta$, 
and \eqref{43bis} on the first line for $\zeta'=\eta^{*,y}$, and on the third  line with $\zeta'=\zeta$.
Combining  \eqref{43plus43bis} and Lemma \ref{lemma_current} yields the result.
\end{proof}
\subsection{Hydrodynamic limits}
It follows from \eqref{eq:theta-lambda} that 
\be\label{mean_rate}
\forall x\in\Z,\,\forall\alpha\in{\bf A},\quad\int_{\bf X}g(\eta(0))d\theta_\lambda(\eta(0))=\lambda=
\int_{\bf X}\alpha(x)g(\eta(x))d\mu^\alpha_\lambda(\eta)
\ee
The quantity
$$
\int_{\bf X}[p\alpha(x)g(\eta(x))-q\alpha(x+1)g[\eta(x+1)]d\mu^\alpha_\lambda(\eta)=(p-q)\lambda
$$
is the stationary current under $\mu^\alpha_\lambda$. As a function of 
the mean density $\overline{R}(\lambda)$, the current writes
\be\label{def_flux}
f(\rho):=(p-q)\overline{R}^{-1}(\rho)
\ee
 We state its basic properties in the following lemma.
\begin{lemma}\label{lemma_properties_flux}
The functions $\overline{R}$  and
$f$ are increasing and $C^\infty$, respectively from $[0,c]$ to $[0,\rho_c]$ and 
from $[0,\rho_c]$ to $[0,c]$.  Besides, $\overline{R}$ is strictly convex if $R$ is strictly convex.
\end{lemma}
\begin{proof}{Lemma}{lemma_properties_flux}
Since $\alpha(.)$ is a bounded sequence with values in $(c,1]$, we can find an infinite subset $I$ of $\N$ and a 
probability measure $Q$ on $[c,1]$ such that the limit
$Q_n\to Q$, as  $n\to+\infty$ in $I$,
holds in the topology of weak convergence, where
$$
Q_n:=\frac{1}{n+1}\sum_{x=-n}^0 \delta_{\alpha(x)}
$$
Since $R\in C^\infty([0,1))$,  for any $\lambda\in [0,c)$, the function 
$\alpha\mapsto R(\lambda/\alpha)$ lies in $C^\infty([c,1])$. Thus
$$
\overline{R}(\lambda)=\int_{(c,1]}R\left(
\frac{\lambda}{\alpha}
\right)Q(d\alpha)
$$
It follows that $\overline{R}\in C^\infty([0,c))$, and
$$
\overline{R}'(\lambda)=\int_{(c,1]}\frac{1}{\alpha}R'\left(
\frac{\lambda}{\alpha}
\right)Q(d\alpha)
$$
Since $R'>0$ on $(0,1)$, the above expression implies $\overline{R}'>0$ on $(0,c)$. 
The same argument applied to $R''$ shows that, if $R$ is strictly convex,
then $\overline{R}$ inherits this property. 
\end{proof}\\ \\
The expected hydrodynamic equation for the limiting density field $\rho(t,x)$ 
of the subcritical disordered zero-range process is
\be\label{conservation_law}
\partial_t \rho(t,x)+\partial_x f[\rho(t,x)]=0
\ee
Hydrodynamic limit of homogeneous asymmetric zero-range processes was established 
by \cite{rez} (see also \cite{kl}).
Convergence of general disordered zero-range processes to the entropy solution of 
\eqref{conservation_law} is proved in \cite{bfl} for subcritical Cauchy data.
For more general Cauchy data (that is, data with certain density values above $\rho_c$), 
the hydrodynamic limit was derived in \cite{ks} in the special case \eqref{special_zrp} 
of the totally asymmetric constant rate model.
The result of \cite{ks} includes the case of a source initial condition, which can 
be viewed as a particular supercritical datum.
For our purpose, we need hydrodynamic limit for a general nearest-neighbour zero-range  
process starting with a source, which does not follow from \cite{bfl} and \cite{ks}.
Besides, we also need a strong local equilibrium statement which, to our knowledge, 
is not available in the disordered or non-convex setting. We recall that strong local 
equilibirum was derived for the homogeneous zero-range process with strictly convex flux in \cite{lan}.
However, the method used there relies on translation invariance of the dynamics, 
which fails in the disordered case. The strategy introduced 
in \cite{bgrs4}, where shift invariance is restored by considering the joint 
disorder-particle process, is not feasible either. Therefore another approach
is required here.\\ \\
The extensions we need are carried out in \cite{bmrs2}, whose results are now stated. 
We consider the  process $(\eta_s^{\alpha,t})_{s\geq 0}$
whose initial configuration is  of the form  (with the convention $(+\infty)\times 0=0$)
\be\label{def_init_source}
\eta^{\alpha,t}_0(x)  =  (+\infty)\indicator{\{x\leq x_t\}}
\ee
This process is a semi-infinite process with a source/sink at $x_t$: 
with rate $p\alpha(x_t)$, a particle is created at $x_t+1$,
with rate $q \alpha(x_t+1)g(\eta(x_t+1))$ a particle at $x_t+1$ is destroyed.\\ \\
For $v\geq 0$, define $\lambda^-(v)$ 
as the smallest
maximizer of  $\lambda\mapsto (p-q)\lambda-v\overline{R}(\lambda)$
over $\lambda\in[0,c]$. Equivalently, $\mathcal R(v):=\overline{R}[\lambda^-(v)]$ is the smallest
maximizer of  $\rho\mapsto f(\rho)-v\rho$
over $\rho\in[0,\rho_c]$.
We also define the Lagrangian, that is the Legendre transform of the current (or Hamiltonian): for $v\in\R$,
\begin{eqnarray}\label{lagrangian}
f^*(v) & := &
\sup_{\rho\in[0,\rho_c]}[f(\rho)-v\rho]
=\sup_{\lambda\in[0,c]}[(p-q)\lambda-v\overline{R}(\lambda)]
\end{eqnarray}
 From standard convex analysis (\cite{roc}), we have that 
\be\label{convex_anal}
\mathcal R(v)=-(f^*)'(v+)
\ee
where $(f^{*})'(v+)$ denotes the right-hand derivative of the convex function $f^*$.
The concave envelope of $f$ is defined by
\be\label{fdoublestar}
\hat{f}(\rho):=f^{**}(\rho):=\inf_{v\in\R}[\rho v+f^*(v)]=\inf_{v\geq 0}[\rho v+f^*(v)]
\ee
The second equality follows from the fact that $f$ is nondecreasing. Indeed, in this case,
 \eqref{lagrangian} implies
that for $v\leq 0$,
$$
f^*(v)=f(\rho_c)-v\rho_c=(p-q)c-v\rho_c
$$
and plugging this into the second member of \eqref{fdoublestar} shows that the infimum 
can be restricted to $v\geq 0$.
It follows from \eqref{convex_anal} that $\mathcal R$ is a nonincreasing and right-continuous function.
\begin{proposition}\label{th_strong_loc_eq} 
Assume that $x_t$ in \eqref{def_init_source} is such that $\beta:=\lim_{t\to+\infty}t^{-1}x_t$, 
with $\beta<0$, and that 
\eqref{assumption_afgl} is satisfied. Then
statement \eqref{hdl_source} below
holds 
for $v\in(0,-\beta]$, 
and statement \eqref{the_first_one} below holds for $v_0<v<-\beta$ and $h:\N^\Z\to\R$ 
 a bounded local increasing function: 
\begin{eqnarray}
\lim_{t\to\infty}\Exp
\left\vert
t^{-1}\sum_{x>x_t+\lfloor vt\rfloor}\eta^{\alpha,t}_t(x)-f^*(v)\right\vert
& = & 0
\label{hdl_source}\\
\label{the_first_one}
\liminf_{t\to\infty}\left\{
\Exp h\left(\tau_{\lfloor x_t+vt\rfloor}\eta_t^{\alpha,t}\right)-\int_{\bf X} h(\eta)d\mu_{\lambda^-(v)}^{\tau_{\lfloor x_t+vt\rfloor}\alpha}(\eta)
\right\}\geq 0
\end{eqnarray}
where $\tau$ denotes the shift operator acting on environments by 
$(\tau_y \alpha)(.)=\alpha(y+.)$ for any $y\in\Z$ and $\alpha\in{\bf A}$.
\end{proposition}
Statement
\eqref{hdl_source}  deals with the hydrodynamics away from the source, while 
\eqref{the_first_one} is a strong local equilibrium statement. Remark that the latter 
differs from the standard strong local equilibrium property in the homogeneous setting, 
since instead of a fixed limiting equilibrium measure, one has an equilibrium measure 
moving along the disorder. We have actually stated in \eqref{the_first_one} only half 
of the complete local equilibrium statement established in \cite{bmrs2} (which also 
includes a reverse inequality and the largest maximizer), because it is sufficient for our current purpose.\\ \\
The values $f^*(v)$ and $\lambda^-(v)$ in \eqref{hdl_source}--\eqref{the_first_one} can be 
understood as follows, see also \cite{bgrs} for a similar variational formula in a different context.
The source process is compared with an equilibrium process with distribution $\mu^\alpha_\lambda$, 
which has asymptotic current $(p-q)\lambda$ and mean density $\overline{R}(\lambda)$. Hence
its current across an observer travelling with speed $v$ is $(p-q)\lambda-v\overline{R}(\lambda)$. 
The source  has a bigger current, which thus
dominates the supremum of these equilibrium currents, that is $f^*(v)$ defined in \eqref{lagrangian}. 
On the other hand, if one admits that the source process around the observer is close to some 
equilibrium process,
then the current must be $(p-q)\lambda-v\overline{R}(\lambda)$ for {\em some} 
(possibly random) $\lambda\in[0,c]$, hence  dominated by $f^*(v)$, and 
this $\lambda$ is a  maximizer  of \eqref{lagrangian}, hence dominating $\lambda^-(v)$. \\ \\
In the sequel, the following quantity will play an important role:
\be\label{def_v0}
v_0:=(p-q)\inf_{\lambda\in [0,c)}\frac{c-\lambda}{\overline{R}(c)-\overline{R}(\lambda)}
\ee
This quantity can be interpreted as the speed of a front  of uniform density $\rho_c$ issued 
by the source. Assumption (H) is equivalent to the infimum in \eqref{def_v0}  being achieved 
uniquely for $\lambda$ tending to $c$, which in turn is equivalent to 
\be\label{equal_v0}
v_0=(p-q)\overline{R}^{'+}(c)^{-1}\in[0,+\infty)
\ee
where $\overline{R}^{'+}$ was defined in \eqref{def_upperdiff}.
Let $\lambda_0$ denote the smallest minimizer of \eqref{def_v0}, or $\lambda_0=c$ 
if the infimum in \eqref{def_v0} is achieved 
only for $\lambda$ tending to $c$, that is under condition (H). The following lemma 
shows that $\overline{R}(\lambda_0)$ is the density observed right behind the front.
\begin{lemma}\label{lemma_entropy}
\mbox{}\\ \\
(i) For every $v<v_0$, $\lambda^-(v)=c$.\\ \\
(ii) 
For every $v>v_0$, $\lambda^-(v)<\lambda_0$, and $\lim_{v\downarrow v_0}\lambda^-(v)=\lambda_0$.
\end{lemma}
\begin{proof}{Lemma}{lemma_entropy}
\mbox{}\\ \\
{\em Proof of (i).} Assume $v<v_0$. Then by definition \eqref{def_v0} of $v_0$, for every $\lambda\in[0,c)$,
$$
(p-q)\lambda-v\overline{R}(\lambda)<(p-q)c-v\overline{R}(c)
$$
Thus $c$ is the unique maximizer in \eqref{lagrangian}, which implies the result.\\ \\
{\em Proof of (ii).}
Assume first that $\lambda^-(v)\geq\lambda_0$ for some $v\geq 0$. Then, for every $\lambda\in[0,\lambda_0)$,
$$
(p-q)\lambda-v\overline{R}(\lambda)\leq(p-q)\lambda^-(v)-v\overline{R}[\lambda^-(v)]
$$
Hence,
$$ 
v\leq (p-q)\inf_{\lambda\in[0,\lambda_0)}\frac{\lambda^-(v)
-\lambda}{\overline{R}[\lambda^-(v)]-\overline{R}(\lambda)}=:v_1
$$
If $\lambda^-(v)=\lambda_0$, then $v_1= v_0$. If $\lambda^-(v)>\lambda_0$,
\eqref{def_v0} implies
$$
(p-q)\frac{\lambda^-(v)-\lambda_0}{\overline{R}[\lambda^-(v)]-\overline{R}(\lambda_0)}\leq v_0
$$
which in turn implies $v_1\leq v_0$.
\\ \\
Let now $(v_n)_{n\in\N\setminus\{0\}}$ be a decreasing sequence of real numbers such 
that $\lim_{n\to+\infty}\lambda^-(v_n)=v_0$, and 
set $\lambda_n:=\lambda^-(v_n)$. The sequence $(\lambda_n)_{n\in\N\setminus\{0\}}$ is nondecreasing 
and nonnegative. The above implies that it is bounded above by $\lambda_0$. Let 
$\lambda\in [0,\lambda_0]$ denote its limit, and assume $\lambda<\lambda_0$. By definition of $\lambda_n$,
$$
(p-q)\lambda_n-v_n\overline{R}(\lambda_n)\geq (p-q)\lambda_0-v_n\overline{R}(\lambda_0)
$$
Passing to the limit as $n\to+\infty$, we obtain the contradiction
$$
v_0\geq(p-q)\frac{\lambda_0-\lambda}{\overline{R}(\lambda_0)-\overline{R}(\lambda)}>v_0
$$
where the strict inequality follows from the fact that $\lambda_0$ is the smallest 
minimizer in \eqref{def_v0}. 
\end{proof}
\section{Proof of Theorem \ref{proposition_afgl}}\label{sec_afgl}
Let $\alpha\in{\bf A}$, $l<0<r$, and $\overline{\eta}_0\in{\bf X}$ a 
configuration  such that $\overline{\eta}_0(x)\in\N$ for 
$x\in(l,r)$, $\overline{\eta}_0(l)=\overline{\eta}_0(r)=+\infty$.
Consider the process $(\overline{\eta}^{\alpha}_t)_{t\geq 0}$, with initial configuration 
$\overline{\eta}_0$, and 
generator $L^{\alpha}$ (see \eqref{generator}). The restriction 
$(\eta_t^{\alpha,l,r})_{t\geq 0}$ of $(\overline{\eta}_t^{\alpha})_{t\geq 0}$ to 
$(l,r)$ is a Markov process on $\N^{(l,r)}$ with generator given by,  
for an arbitrary (since $g$ is bounded)  function $f$ on  $\N^{(l,r)}$,
\begin{eqnarray}\nonumber
L^{\alpha,l,r}f(\eta)&=&\sum_{x=l+1}^{r-2}
p\alpha(x)g(\eta(x))[f(\eta^{x,x+1})-f(\eta)]\\
\nonumber & + & \sum_{x=l+1}^{r-2} q\alpha(x+1)g(\eta(x+1))[f(\eta^{x+1,x})-f(\eta)]\\
\nonumber & + & qg(\eta(l+1))[f(\eta-\delta_{l+1})-f(\eta)] \\ \nonumber
\nonumber & + & pg(\eta(r-1))[f(\eta-\delta_{r-1})-f(\eta)]\\ \nonumber 
\nonumber &+ & p\alpha(l)[f(\eta+\delta_{l+1})-f(\eta)]\\
\label{gen_open} & + & q\alpha(r)[f(\eta+\delta_{r-1})-f(\eta)]
\end{eqnarray} 
The above process is an open Jackson network, whose invariant measure is 
well-known in queuing theory. In our case this measure is explicit:
\begin{lemma}\label{lem-mu_c'}
Set
\be\label{def_lambda_inv}
\lambda^{\alpha,l,r}(x)=\frac{\alpha(r)-\alpha(l)}{1-\left(\frac{q}{p}\right)^{r-l}}
\left(\frac{q}{p}\right)^{r-x}+
\frac{\alpha(l)-\alpha(r)\left(\frac{q}{p}\right)^{r-l}}{1-\left(\frac{q}{p}\right)^{r-l}}
\in[\alpha(l),\alpha(r)]
\ee
The process with generator \eqref{gen_open} is positive recurrent if and only if 
\be\label{cond_rec}\lambda^{\alpha,l,r}(x)<\alpha(x),\quad\forall x\in(l,r)\cap\Z\ee
If condition \eqref{cond_rec} is satisfied, the unique invariant measure of 
the process is the product measure
$\mu^{\alpha,l,r}$ on $\N^{(l,r)\cap\Z}$ with marginal 
$\theta_{\lambda(x)/\alpha(x)}$ at site $x\in(l,r)\cap\Z$.
\end{lemma}
\begin{proof}{Lemma}{lem-mu_c'}
An explicit computation shows that $\lambda^{\alpha,l,r}(.)$ given in  
\eqref{def_lambda_inv} is the unique solution to the following system:
\begin{eqnarray}
\lambda(x) & = & p\lambda(x-1)+q\lambda(x+1)\mbox{ if } x\in (l+1,r-1)\cap\Z\label{jackson_bulk}\\
\lambda(l+1) & = & p\alpha(l)+q\lambda(l+2)\label{jackson_left}\\
\lambda(r-1) & = & p\lambda(r-2)+q\alpha(r)\label{jackson_right}
\end{eqnarray}
A standard result in queuing theory (see e.g. \cite{par}) states that  the process 
with generator \eqref{gen_open} is positive recurrent if and only if this solution 
satisfies condition \eqref{cond_rec}, and that in this case, it  has  as unique invariant measure $\mu^{\alpha,l,r}$. 
\end{proof}
\mbox{}\\ \\
We can now conclude the\\ \\
{\em Proof of the upper bound \eqref{upper_bound}.} Recall that 
 the quenched process $(\eta^\alpha_t)_{t\geq 0}$ 
in \eqref{upper_bound} has  initial 
configuration $\eta_0\in\N^\Z$, and  generator \eqref{generator}.
For $\varepsilon>0$, let 
\begin{eqnarray}\label{def:ell}
A_\varepsilon:=A_\varepsilon(\alpha)&=&\max\{x\le 0:\alpha(x)\leq c+\varepsilon\}\\
a_\varepsilon:=a_\varepsilon(\alpha)&=&\min\{x\ge 0:\alpha(x)\leq c+\varepsilon\}\label{def:err}
\end{eqnarray}
 We can regard $a_\varepsilon$ and $A_\varepsilon$ as positions of potential bottlenecks since the flux across these points is
close to the maximum uniformly on all configurations. Thanks to assumption \eqref{not_too_sparse}, the set in \eqref{def:ell} is never empty.
In contrast, the set in \eqref{def:err} may be empty, in which case, by the usual convention, $a_\varepsilon(\alpha)$ is set to $+\infty$.
It follows from definition \eqref{def:ell} that
\be\label{location}
\lim_{\varepsilon\to 0}A_\varepsilon=-\infty
\ee 
Set
\be\label{def:l-r}
 l:=A_\varepsilon(\alpha),\quad r:=a_\varepsilon(\alpha)\indicator{\{a_\varepsilon(\alpha)<+\infty\}}
+\varepsilon^{-1}\indicator{\{a_\varepsilon(\alpha)=+\infty\}}
 \ee
We define $r'\in\Z$, $\alpha'\in{\bf A}$, a $[0,c]$-valued function $\lambda_\varepsilon(.)$ 
on $(l,r')$ and a probability measure 
$\mu_\varepsilon$ on $\N^{(l,r')}$ as follows,  so that $\mu_\varepsilon$ 
is an invariant measure for $L^{\alpha',l,r'}$. \\ \\
{\em First case.}
Condition \eqref{cond_rec} is satisfied, thus the measure $\mu^{\alpha,l,r}$ of 
Lemma \ref{lem-mu_c'} is well defined.
This is true in particular if $a_\varepsilon(\alpha)<+\infty$, see \eqref{def_lambda_inv}. 
We set $r':=r$, $\alpha'=\alpha$, $\lambda_\varepsilon(.)=\lambda^{\alpha,l,r}(.)$ and $\mu_\varepsilon:=\mu^{\alpha,l,r}$.
\\ \\
{\em Second case.}
Condition \eqref{cond_rec} is not satisfied. We then set
\be\label{def_right_reservoir}
r'=r'(\alpha,l,r):=\min\{x>l:\,\lambda^{\alpha,l,r}(x)>\alpha(x)\}
\ee
and define a modified environment $\alpha'=\alpha'(\alpha,l,r)$ by setting
\be\label{modif_alpha}
\alpha'(x)=\left\{
\ba{lll}
\alpha(x) & \mbox{if} & x\neq r'\\
\lambda^{\alpha,l,r}(x) & \mbox{if} & x=r'
\ea
\right.
\ee
Since $\alpha(.)$ and $\lambda^{\alpha,l,r}(.)$  satisfy 
\eqref{jackson_bulk}--\eqref{jackson_right}, by construction, $\alpha'(.)$  and the restriction of 
$\lambda^{\alpha,l,r}(.)$ 
to $(l,r')\cap\Z$ still satisfy these equalities, and we define $\lambda_\varepsilon(.)$ 
as this restriction. Thus $\mu_\varepsilon:=\mu^{\alpha',l,r'}$, is an invariant measure 
for $L^{\alpha',l,r'}$.  \\ \\
Define
the initial configuration
$\overline{\eta}_0\in{\bf X}$  by $\overline{\eta}_0(x)=\eta_0(x)$ 
for all $x\not\in\{l,r'\}$, $\overline{\eta}_0(l)=\overline{\eta}_0(r')=+\infty$.
As above for Lemma \ref{lem-mu_c'}, $(\overline{\eta}_t^{\alpha})_{t\ge 0}$ 
denotes the process with generator \eqref{generator} and initial configuration 
$\overline{\eta}_0$,  and $(\eta^{\alpha,l,r'}_t)_{t\geq 0}$ its restriction 
to $(l,r')$. Recall from the above preliminary that $(\eta^{\alpha,l,r'}_t)_{t\geq 0}$ is
a Markov process
with generator $L^{\alpha,l,r'}$ defined by \eqref{gen_open}.
Since $\eta_0\leq\overline{\eta}_0$, by attractiveness, we have $\eta_t^\alpha\leq\overline{\eta}_t^\alpha$.
Let
$(\eta_t^{\alpha',l,r'})_{t\ge 0}$ be the process with generator $L^{\alpha',l,r'}$ 
defined by \eqref{gen_open}, and whose initial configuration 
is the restriction of $\eta_0$ to $(l,r')$. 
This process converges in distribution as $t\to+\infty$ to its invariant measure 
$\mu_\varepsilon$ defined above.
By attractiveness, and the fact that the entrance rate $q\alpha(r')$ at $r'$ in 
 $L^{\alpha',l,r'}$ has been increased with respect to that of $L^{\alpha,l,r'}$, 
 we have that $\eta_t^{\alpha,l,r'}\leq\eta_t^{\alpha',l,r'}$.\\ \\
{}From \eqref{def_lambda_inv} it follows that $\lambda_\varepsilon(x)$ in a finite neighbourhood of 0
can be made arbitrarily close to $\alpha(l)$ by choosing $\varepsilon$ appropriately small. This
in turn implies that $r'$ goes to infinity as $\varepsilon$ goes to 0 in both first and second cases.
Thus for $\varepsilon$ small enough, the support of $h$ is contained in $(l,r')$.
Since $h$ is nondecreasing with support in $(l,r)$, we then have
\begin{eqnarray*}
\label{eq:basdepage}
\limsup_{t\to \infty} \Exp h(\eta_t^\alpha) & \le & \lim_{t\to \infty} \Exp h(\overline{\eta}_t^\alpha)= \lim_{t\to \infty} \Exp h(\eta_t^{\alpha,l,r'})\\
& \leq & \lim_{t\to \infty} \Exp h(\eta_t^{\alpha',l,r'})=\int_{\bf X} h(\eta)d\mu_\varepsilon(\eta)
\end{eqnarray*}
Since as $\varepsilon$ goes to 0, $\alpha(A_\varepsilon)$ goes to $c$, we have 
 that $\lambda_\varepsilon(x)$ converges uniformly to $c$ 
on every finite subset of $\Z$.  Hence,
$$
\lim_{\varepsilon\to 0}\int_{\bf X} h(\eta)d\mu_\varepsilon(\eta)=\int_{\bf X} h(\eta)d\mu^\alpha_c(\eta)
$$
This concludes the proof of Theorem \ref{proposition_afgl}. \hfill\mbox{$\square$}
\section{Proof of Theorem \ref{prop_nec}}\label{sec_nec}
Let $y\in\Z\cap(-\infty,0)$.
By Corollary \ref{corollary_consequence} (with $z=0$ and $\zeta=\eta_0$),
\be\label{semi_vari}
\Gamma_0(t,\eta_0)\leq\Gamma_0(t,\eta^{*,y})+\sum_{x=y+1}^0\eta_0(x)
\ee   
Let $z\in(0,+\infty)$, and $(t_n)_{n\in\N}$ denote a sequence of nonnegative real numbers such that
$$
\lim_{n\to+\infty}t_n=+\infty,\quad
\lim_{n\to\infty}(t_n z)^{-1}\sum_{x=[-t_n z]}^0\eta_0(x)=\rho
$$
where $\rho<\rho_c$ 
is given in \eqref{eq_non16}. 
Let  $y_n:=\lfloor -t_n z\rfloor$. 
Taking quenched expectation of \eqref{semi_vari} yields
\be\label{quenched_exp}
\Exp\left\{t_n^{-1}\Gamma_0(t_n,\eta_0)\right\}\leq\Exp\left\{t_n^{-1}\Gamma_0\left(t_n,
\eta^{*,y_n}\right)\right\}+t_n^{-1}\sum_{x=\lfloor-t_n z\rfloor}^0\eta_0(x)
\ee
By  Proposition \ref{th_strong_loc_eq},
\be\label{density-flux}
\lim_{n\to+\infty}\Exp\left\{t_n^{-1}\Gamma_0\left(t_n,\eta^{*,y_n}\right)\right\}=f^*(z)
\ee
where $f^*$ is the Legendre transform of $f$ defined by \eqref{lagrangian}.
Passing  to the limit in \eqref{quenched_exp} as $n\to+\infty$,
we obtain
$$
\limsup_{n\to+\infty}\Exp\left\{t_n^{-1}\Gamma_0(t_n,\eta_0)\right\}\leq \rho z+f^*(z)
$$
Since the above is true for every $z>0$, we have 
\be\label{upper_bound_current}
\liminf_{t\to+\infty}\Exp\left\{t^{-1}\Gamma_0(t,\eta_0)\right\}\leq\inf_{z>0}[ \rho z+f^*(z)]=\hat{f}(\rho)
\ee
where $\hat{f}$ is the concave envelope of $f$ defined in \eqref{fdoublestar}. Note that the infimum in \eqref{upper_bound_current} is equal to the one
in \eqref{fdoublestar} by continuity of $f^*$. Since $f$ is strictly increasing, we have that $\hat{f}(\rho)<\hat{f}(\rho_c)=f(\rho_c)=(p-q)c$.
We have thus shown that
\be\label{subcritical_current}
\liminf_{t\to+\infty}\Exp\left\{t^{-1}\Gamma_0(t,\eta_0)\right\}<(p-q)c
\ee  
Now assume that the conclusion of the proposition is false, i.e. that 
$\eta_t^\alpha$ converges in distribution to $\mu^\alpha_c$ as $t\to+\infty$.
Since  
$$
\Exp\left\{t^{-1}\Gamma_0(t,\eta_0)\right\}=t_n^{-1}\int_0^{t}\Exp\left\{
p\alpha(0)g(\eta_s(0))-q\alpha(1)g(\eta_s(1))
\right\}ds
$$
it would follow that
$$
\lim_{t\to+\infty}\Exp\left\{t^{-1}\Gamma_0(t,\eta_0)\right\}=\int\left\{
p\alpha(0)g(\eta(0))-q\alpha(1)g(\eta(1))
\right\}d\mu^\alpha_c(\eta)=(p-q)c
$$
 which contradicts \eqref{subcritical_current}.
\section{Proof of Theorem \ref{prop_counter}}\label{sec:proof_counter}
By assumption  \eqref{not_too_sparse},  there exists a decreasing sequence 
$(X_k)_{k\in\N}$   of negative integers such that
\be\label{limalpha}
\lim_{k\to+\infty}\alpha(X_k)=c
\ee
Given this sequence, we construct a decreasing sequence $(x_n)_{n\in\N}$ of negative integers
such that $(x_n)_{n\in\N}$ is a subsequence of $(X_k)_{k\in\N}$, and 
\be\label{lim_an}\lim_{n\to+\infty}\frac{x_n}{\delta_n}=0
\ee
where 
\be\label{induction_2}\delta_n:=x_{n+1}-x_n\ee
Set $x_0=0$ and pick $t_0>0$, then set $\delta_0=(1+V)t_0$, where $V$ is a 
finite propagation speed constant given for this kernel $p(.)$ by Lemma
\ref{lemma_finite_prop}.
Then for every $n\in\N$, define 
\begin{eqnarray}
\label{induction_1} t_{n+1} & = & \frac{\delta_{n+1}}{1+V}
\end{eqnarray}
Let $\eta_0\in{\bf X}$ be the particle configuration defined by 
\be\label{def_conf}\eta_0(x)=\left\{
\ba{lll}
0 & \mbox{ if } & x\not\in\{x_n,\,n\geq 1\}
\ea
\right.
\ee
where the sequence $(y_n)_{n\in\N}$ is defined as follows: 
\be\label{def_spike}
y_n=\left\{
\ba{lll}
\lfloor \rho_c \delta_n\rfloor & \mbox{if} & \rho_c<+\infty\\
\lfloor \rho_n \delta_n \rfloor & \mbox{if} & \rho_c=+\infty
\ea
\right.
\ee
and $(\rho_n)_{n\in\N}$ (in the case $\rho_c=+\infty$) is a sequence satisfying
\be\label{cond_spike}
\rho_n\in [0,+\infty),\quad\lim_{n\to+\infty}\rho_n=+\infty,\quad \lim_{n\to+\infty}\rho_n\frac{x_n}{t_n}=0
\ee
Let $\eta^n_0$ be the truncated particle configuration defined by
\be
\label{def_eta_n}\eta^n_0(x)=\left\{
\ba{lll}\eta_0(x) & \mbox{ if } & x\geq x_n\\ 
0  & \mbox{ if } & x<x_n
\ea
\right.
\ee
We denote respectively by $\eta_t^\alpha$ and $\eta^{\alpha,n}_t$ 
the evolved zero-range processes starting from $\eta_0$ and $\eta^n_0$.
By finite propagation, \eqref{induction_2} and \eqref{def_eta_n}, 
there exists a vanishing sequence $(p_n)_{n\in\N}$ with values in $[0,1]$, 
such that the $\Prob$-probability of the event 
$$ 
\eta_t^\alpha(x)=\eta^{\alpha,n}_t(x),\quad \forall t\in[0,t_n],\,\forall x\geq x_n
$$
is bounded below by $1-p_n$, uniformly with respect to $\alpha\in{\bf A}$.
On the other hand, for $\eta^n_t$, because $p(.)$ is totally asymmetric 
and there are initially no particles to the left of $x_n$, we have the following bound:
\be\label{bound_eta_n}
\Gamma_0(t_n,\eta^n_0)\leq N^{x_n}_{t_n}+\sum_{x>x_n}\eta^n_0(x)=N^{x_n}_{t_n}+\sum_{x> x_n}\eta_0(x)
\ee
where $(N^x_t)_{t\geq 0}$ denotes the Poisson process of potential jumps 
to the right from $x$ in the Harris construction, that is a Poisson process with rate 
$$\lambda(x):=\alpha(x)\sum_{z>0}p(z)=\alpha(x)$$
Note that definitions \eqref{def_conf}--\eqref{def_spike} and assumption \eqref{cond_spike} imply
\be\label{limit_density}
\lim_{n\to+\infty}\frac{1}{x_n}\sum_{x>x_n}\eta_0(x)=\lim_{n\to+\infty}\frac{1}{x_n}\sum_{i=0}^{n-1} y_i=\rho_c,
\ee
Hence the supercriticality condition \eqref{cond_init} is satisfied by $\eta_0$.
The bound \eqref{bound_eta_n}, combined with the law of large numbers for  Poisson processes,  
\eqref{lim_an}, \eqref{induction_1},  \eqref{limit_density} and \eqref{limalpha}, yields that the limit 
\be\label{limcurrent_n}
\limsup_{n\to+\infty}\Exp\left\{
t_{n}^{-1}\Gamma_0(t_{n},\eta^n_0)
\right\}\leq \lim_{n\to+\infty}\Exp\left\{
t_{n}^{-1}N^{x_{n}}_{t_{n}}
\right\}=c\sum_{z>0}p(z)=c
\ee
holds $\Prob$-a.s. The finite propagation lemma applied to $\eta_t$ and $\eta^n_t$ on $[0,t_n]$, 
together with \eqref{limcurrent_n}, implies 
$$ 
\limsup_{n\to+\infty}\Exp\left\{
t_{n}^{-1}\Gamma_0(t_{n},\eta_0)
\right\} \leq c\sum_{z>0}p(z)=c
$$
But 
\be\label{exp_current}
\Exp\left\{
t_{n}^{-1}\Gamma_0(t_{n},\eta_0)
\right\} = t_{n}^{-1}\int_0^{t_{n}}
\sum_{x\leq 0}\sum_{z>-x}
p(z)\alpha(x)g\left(\eta^\alpha_t(x)\right)
dt
\ee
 It follows that $\eta^\alpha_t$ cannot converge in distribution to 
 $\mu^\alpha_c$ as $t\to+\infty$, for if this were true, the r.h.s. of \eqref{exp_current} would converge to
$$
\int_{\bf X}
\sum_{x\leq 0}\sum_{z>-x}
p(z)\alpha(x)g\left(\eta(x)\right)
d\mu^\alpha_c(\eta)=c\sum_{z\in\Z}zp(z)>c$$
\section{The lower bound}\label{sec_ideas}
The lower bound of Proposition \ref{prop_main_work} can be established 
along the following lines, see \cite{bmrs2} for (lengthy) details of this scheme. 
The supercriticality condition \eqref{cond_init} is shown in a first step to imply 
that locally around the origin, our process $(\eta_s^\alpha)_{s\geq 0}$ dominates 
the source process $(\eta^{\alpha,t,\beta}_s)_{s\geq 0}$
with initial configuration defined by \eqref{def_init_source}, with 
$x_t:=\lfloor\beta t\rfloor$, and $\beta<0$. This holds in the following sense.
Recalling the definition  \eqref{def:ell} of $A_\varepsilon$, we have that 
\begin{lemma}\label{lemma_asymptotic_order}
For any $\beta<-v_0$, 
$$
\lim_{\varepsilon\to 0}\liminf_{t\to+\infty}\Prob\left(
\left\{
\eta_t^\alpha(x)\geq\eta^{\alpha,t,\beta}_t(x),\,\forall x\geq A_\varepsilon(\alpha)
\right\}
\right)=1
$$
\end{lemma}
Due to \eqref{location}, Lemma \ref{lemma_asymptotic_order} implies domination 
in any fixed neighbourhood of the origin.
The validity of Lemma \ref{lemma_asymptotic_order} does not depend on the actual 
choice of $\beta<-v_0$, but this choice does matter to ensure
that the comparison source process itself is close to the critical measure 
$\mu^\alpha_c$ near the origin. This latter property, combined with 
Lemma \ref{lemma_asymptotic_order}, will imply Proposition \ref{prop_main_work}. 
It turns out that the relevant choice is to take $\beta$ arbitrarily close to $-v_0$. Precisely, we have that
\begin{lemma}
\label{lemma_asymptotic_source}
For any bounded local increasing function $h:\N^\Z\to\R$,
$$
\lim_{\beta\to -v_0}\liminf_{t\to+\infty}\Exp h(\eta^{\alpha,t,\beta}_t)\geq\int_{\bf X}h(\eta)d\mu^\alpha_c(\eta)
$$
\end{lemma}
\begin{proof}{Lemma}{lemma_asymptotic_source} For given $\beta<-v_0$, by Proposition \ref{th_strong_loc_eq},
$$
\liminf_{t\to+\infty}\Exp h(\eta^{\alpha,t,\beta}_t)\geq\int_{\bf X}h(\eta)d\mu^\alpha_{\lambda^-(-\beta)}(\eta)
$$
Letting $\beta\to -v_0$, by statement (ii) of Lemma \ref{lemma_entropy}, we have
$$
\lim_{\beta\uparrow -v_0}\int_{\bf X}h(\eta)d\mu^\alpha_{\lambda^-(-\beta)}(\eta)=\int_{\bf X}h(\eta)d\mu^\alpha_{\lambda_0}(\eta)
$$
But Assumption (H) is exactly the necessary and sufficient condition to have $\lambda_0=c$.
\end{proof}\\ \\
Proposition \ref{prop_main_work} then follows from Lemmas \ref{lemma_asymptotic_order} and \ref{lemma_asymptotic_source}.\\ \\
\noindent {\bf Acknowledgments:}
This work was partially supported by ANR-2010-BLAN-0108, PICS 5470, FSMP,
and grant number 281207 from the
Simons Foundation awarded to K. Ravishankar.
We thank 
Universit\'{e}s Blaise Pascal and Paris Descartes, EPFL
and SUNY College at New Paltz, for hospitality.
\end{document}